\documentclass[10pt]{amsart}

\usepackage{mathtools, amssymb, amsthm, tikz-cd}
\usepackage[hidelinks]{hyperref}

\theoremstyle{plain}
\newtheorem{theorem}{Theorem}[section]
\newtheorem{proposition}[theorem]{Proposition}
\newtheorem{lemma}[theorem]{Lemma}
\newtheorem{corollary}[theorem]{Corollary}

\theoremstyle{definition}
\newtheorem{remark}[theorem]{Remark}

\numberwithin{equation}{section}

\title{Negative resolution to the $C^*$-algebraic Tarski problem}

\author{Srivatsav Kunnawalkam Elayavalli}
\address{Srivatsav Kunnawalkam Elayavalli, 
Department of Mathematics, University of California, San Diego, 9500 Gilman Drive \#0112, La Jolla, CA 92093, USA}
\email{skunnawalkamelayaval@ucsd.edu}

\author{Christopher Schafhauser}
\address{Christopher Schafhauser, Department of  Mathematics, University of Ne\-bra\-ska--Lincoln, 1400 R St., Lincoln, NE 68588, USA}
\email{cschafhauser2@unl.edu}

\date{\today}
\thanks{This project was partially supported by NSF grants DMS 2350049 (Kunnawalkam Elayavalli) and DMS 2400178 (Schafhauser).}

\begin{document}
	
\begin{abstract}
	We compute the $K_1$-group of ultraproducts of unital, simple $C^*$-algebras with unique trace and strict comparison.  
    As an application, we prove that the reduced free group $C^*$-algebras $C^*_r(F_m)$ and $C^*_r(F_n)$ are elementarily equivalent (i.e., have isomorphic ultrapowers) if and only if $m = n$. This settles in the negative the $C^*$-algebraic analogue of Tarski's 1945 problem for groups.
\end{abstract}

\maketitle

\section{Introduction}


The structure of non-abelian free groups $F_n$ has been a subject of sustained interest over the last century. In this avenue, a landmark problem of the twentieth century was formulated by Tarski in 1945, which asks if these groups are pairwise elementarily equivalent, i.e., if they have the same first-order theory as groups. By the Keisler--Shelah theorem, this is equivalent to asking if there are ultrafilters  $\omega_1$ and  $\omega_2$ on some sets $I_1$ and $I_2$, respectively, such that the ultrapowers $({F_n})_{\omega_1}$ and $ ({F_m})_{\omega_2}$ are isomorphic (see Section \ref{sec:ultraproducts} for definitions).  This problem remained open for several decades until the remarkable works of Sela (\cite{TarskiSela}) and Kharlampovich--Myasnikov (\cite{KHARLAMPOVICH2006451}) independently resolved this problem in the affirmative. While the free groups themselves are pairwise non-isomorphic, with their abelianizations $\mathbb{Z}^n$ distinguished by rank, the first-order theory completely forgets this information.

In the context of functional analytic completions of groups, leading to group operator algebras, there is a natural analogue of Tarski's problem: namely, if the reduced $C^*$-algebras of the free groups are elementarily equivalent in the first-order theory of $C^*$-algebras. This problem has been circulated among experts and has been of importance in the field since the inception of the systematic study of model theory of operator algebras and, in particular, $C^*$-algebras (see \cite{mtoa1, MTOA2, MTOA3, FHLRTVW:MAMS}). In this paper, we settle this problem in the negative. This surprisingly goes in the opposite direction of the resolution of Tarski's problem in the group case.

\begin{theorem}\label{thm:tarski}
	If $m, n \in \mathbb N \cup \{\infty\}$ are such that $C^*_r(F_m)$ and $C^*_r(F_n)$ are elementarily equivalent, then $m = n$.
\end{theorem}

It is a result of Pimsner and Voiculescu (\cite{PimsnerVoiculescu:JOT}) that $K_1(C^*_r(F_n)) \cong \mathbb Z^n$ for $n \in \mathbb N \cup \{\infty\}$, and hence the reduced free group $C^*$-algebras are mutually non-isomorphic.  It is also easy to show that if $m, n \in \mathbb N \cup \{\infty\}$ are such that $\mathbb Z^m$ and $\mathbb Z^n$ are elementarily equivalent, then $m = n$.  The difficulty in using this to prove Theorem~\ref{thm:tarski} is that for a $C^*$-algebra $A$, the elementary equivalence class of $K_1(A)$ is generally not an elementary invariant of $A$. In more detail, for a $C^*$-algebra $A$ and a free ultrafilter $\omega$ on $\mathbb N$, there is a natural group homomorphism
\begin{equation}\label{eq:bareta-intro}
	\bar\eta \colon K_1(A_\omega) \rightarrow K_1(A)_\omega,
\end{equation}
where $A_\omega$ denotes the $C^*$-algebra ultrapower of $A$ and $K_1(A)_\omega$ denotes the (discrete) abelian group ultrapower of $K_1(A)$ (see Section~\ref{sec:ultraproducts}), but $\bar\eta$ is typically not an isomorphism.  In fact, both injectivity and surjectivity fail in general.

The fundamental problem is that $K_1(A)$ considers unitaries and homotopies between them in all matrices over (the minimal unitization of) $A$, and in general, there is no bound on the size of matrices needed to witness elements of $K_1(A)$ or homotopies between unitaries.  However, in the context of Theorem~\ref{thm:tarski}, when $A$ is a reduced free group algebra, this problem can be avoided.  Indeed, in this case, $A$ has stable rank one (i.e., a dense set of elements are invertible) by a result of Dykema, Haagerup, and R{\o}rdam (\cite{DykemaHaagerupRordam:Duke}), and Rieffel has shown in \cite{Rieffel:JOT} that for a unital $C^*$-algebra $A$ of stable rank one, $K_1(A) = U(A)/ U^0(A)$, where $U(A)$ is the unitary group of $A$ with the norm topology and $U^0(A) \subseteq U(A)$ is the path-component containing $1_A$.  So, no matrix amplifications are necessary in the definition of $K_1(A)$.

It follows immediately from this result that $\bar\eta$ is surjective for unital $C^*$-algebras with stable rank one, but another issue arises when using this to prove injectivity: $U^0(A)$ need not be definable.  In more concrete terms, it is a standard fact that if $u \in U^0(A)$, then there are an integer $n \geq 1$ and self-adjoint elements $h_1, \ldots, h_n \in A$ of norm at most $\pi$ such that 
\begin{equation}\label{eq:exponentials}
	u = e^{i h_1} \cdots e^{i h_n},
\end{equation}
but it is generally not possible to bound $n$ uniformly over all $u \in U^0(A)$ (see \cite{Phillips:AmerJMath}).  If $A$ is a unital $C^*$-algebra with stable rank one and a uniform bound on $n$ exists, it is easy to show  that $\bar\eta$ is injective and hence is an isomorphism.

The notions of \emph{exponential rank} and \emph{exponential length} of a $C^*$-algebra (see \cite{PhillipsRingrose:Nara, Phillips:MathScand} and \cite{Ringrose:Edinburgh}, respectively) were developed precisely to study when such bounds exist.  While we have not been able to produce a uniform bound on $n$ for reduced free group algebras $A$, we show that such decompositions do exist after a $2\times2$ matrix amplification (see Theorem~\ref{thm:k1-inj}).  This weaker result is sufficient to prove $\bar\eta$ is an isomorphism and hence to prove Theorem~\ref{thm:tarski}.

A major ingredient is the recent result of the first author with Amrutam, Gao, and Patchell showing that reduced free group $C^*$-algebras have strict comparison (\cite{AGKEP}, building on \cite{louder2022strongly, robert2023selfless}).  Then Theorem~\ref{thm:tarski} follows from the more general result below.  

\begin{theorem}\label{thm:k1-inj}
	Let $A$ be a unital, simple $C^*$-algebra with unique trace and strict comparison.  If $u \in A$ is a unitary with $[u]_1 = 0 \in K_1(A)$, then for every $\epsilon > 0$, there are self-adjoint $h, k \in M_2(A)$ with $\|h\|, \|k\| \leq \pi$ and $\| u \oplus 1_A - e^{i h} e^{i k} \| < \epsilon$.
\end{theorem}

For $A$ as in Theorem~\ref{thm:k1-inj}, a recent result of Lin (\cite{Lin:sr1}) states that $A$ has stable rank one (see Theorem~\ref{thm:sr1} below).  This then implies that the map $\bar\eta$ in \eqref{eq:bareta-intro} is an isomorphism.  In fact, we also show that the analogous maps are isomorphisms for products and ultraproducts of sequences of $C^*$-algebras as in Theorem~\ref{thm:k1-inj}: see Theorem~\ref{thm:k1-product} and Corollary~\ref{cor:k1-ultraproduct}, respectively.

We end this introduction with a discussion of the methods used in Theorem~\ref{thm:k1-inj}.
There has been significant work on variations of this kind of decomposition dating back to the early 90s (see \cite{PhillipsRingrose:Nara, Phillips:MathScand, Ringrose:Edinburgh, Zhang:AnnMath, GongLin:MathScand,Lin:JFA1}, for example).  A very general result of this form is due to Lin in \cite{Lin:JFA2} for simple, $\mathcal Z$-stable $C^*$-algebras which, after tensoring with a UHF algebra, have a rich internal tracial approximation structure.  The basic strategy used in \cite{Lin:JFA2}, which is also used here, is to classify unital $^*$-homomorphisms $C(\mathbb T) \rightarrow A$ up to approximate unitary equivalence and to exhaust the range of the invariant with unitaries of a particular form to conclude that unitaries in $U^0(A)$ are approximately conjugate to a product of exponentials of the desired form.

In the case of the reduced free group algebras, and many other $C^*$-algebras covered by Theorem~\ref{thm:k1-inj}, neither $\mathcal Z$-stability nor the internal tracial approximation conditions hold.  We avoid the need for tracial approximations using the trace-kernel methods developed in \cite{CGSTW1,CGSTW2} (refining those of \cite{Schafhauser:Crelle,Schafhauser:AnnMath}) in relation to much more general classification results for embeddings of $C^*$-algebras.  While the results do not apply directly (most prominently, due to the lack of $\mathcal Z$-stability of the codomain), many of the intermediate results and techniques do apply.

As noted in \cite[Section~1.3.6]{CGSTW1}, $\mathcal Z$-stability essentially plays three roles in the classification results of \cite{CGSTW1}.  One is to obtain strict comparison, which we assume to hold in the setting of Theorem~\ref{thm:k1-inj} and holds in the setting of Theorem~\ref{thm:tarski} by \cite{AGKEP}.  Another is to access the ``complemented partitions of unity'' techniques from \cite{CETWW:InventMath}, which are irrelevant in the present setting of $C^*$-algebras with unique trace.  Finally, $\mathcal Z$-stability more subtly appears as a way of removing a matrix amplification occurring in the proof of the stable uniqueness theorem of Dadarlat and Eilers (particularly, \cite[Theorem~3.8]{DadarlatEilers:KT}); see \cite[Theorem~5.15]{CGSTW1}.  This last use of $\mathcal Z$-stability can be avoided at the expense of passing to the $2 \times 2$ matrix amplification in  Theorem~\ref{thm:k1-inj}.

\subsection*{Acknowledgments} This work began during the second author's visit to UCLA in the Fall of 2022, and he thanks the first author and the UCLA math department for their hospitality. We also thank Bradd Hart, Ilijas Farah, David Jekel, Michael Magee, and Aaron Tikuisis for helpful comments on the paper. 
\section{Preliminaries}

Throughout, $\omega$ will denote a free ultrafilter on the natural numbers.  For a $C^*$-algebra $A$ and an integer $d \geq 1$, let $M_d(A)$ denote the $C^*$-algebra of $d \times d$ matrices over $A$.  We identify $M_d(A)$ as a subalgebra of $M_{d+1}(A)$ along the embedding $a \mapsto a \oplus 0$ and write $M_\infty(A)$ for the union.  For integers $d_1, d_2 \geq 1$ and $a \in M_{d_1}(A)$ and $b \in M_{d_2}(A)$, let $a \oplus b \in M_{d_1+d_2}(A)$ denote the block diagonal matrix with diagonal entries $a$ and $b$.  In this way, we also define $a \oplus b \in M_\infty(A)$ for all $a, b \in M_\infty(A)$.  For a $^*$-homomorphism $\phi \colon  A \rightarrow B$ between $C^*$-algebras $A$ and $B$ and an integer $d \geq 1$, the induced map $M_d(A) \rightarrow M_d(B)$ will also be denoted by $\phi$.  A \emph{trace} $\tau$ on a $C^*$-algebra $A$ will always mean a tracial state.  Note that any such $\tau$ induces a positive tracial functional on $M_\infty(A)$, still denoted $\tau$, given by $a \mapsto \sum_{i=1}^\infty \tau(a_{i, i})$.  

We recall the notions of Cuntz subequivalence and strict comparison in Section~\ref{sec:comparison}.  Our notation for ultraproducts and the definition of the map $\bar\eta$ in \eqref{eq:bareta-intro} are set out in Section~\ref{sec:ultraproducts}.  Section~\ref{sec:trace-kernel} recalls the definition and some properties of the trace-kernel extension of a $C^*$-algebra with respect to a trace, and finally, Section~\ref{sec:absorption} contains some remarks on absorbing $^*$-homomorphisms into stable multiplier algebras.

\subsection{Cuntz subsequivalence and comparison}\label{sec:comparison}

Cuntz subequivalence was defined in \cite{Cuntz:MathAnn} in connection with the study of dimension functions and traces on $C^*$-algebras and has become a indispensable tool in modern $C^*$-algebra theory.  The $C^*$-algebraic comparison theory has its origins in the work of Blackadar in \cite{Blackadar:LMS}, motivated by the Murray--von Neumann comparison theory for projections in von Neumann algebras, and was rephrased in terms of the ordered abelian semigroup property known as \emph{almost unperforation} by R{\o}rdam in \cite{Rordam:IntJMath}.  We recall some of the basic notions here.

For a $C^*$-algebra $A$ and positive elements $a, b \in M_\infty(A)$, we say $a$ is \emph{Cuntz subequivalent} to $b$ and write $a \precsim b$ if there is a sequence $(v_n)_{n=1}^\infty \subseteq M_\infty(A)$ such that $\|v_n^* b v_n - a\| \rightarrow 0$.   We say $A$ has \emph{almost unperforated Cuntz semigroup} if whenever $k \geq 1$ is an integer, $a, b \in M_\infty(A)$ are positive, and $a^{\oplus k + 1} \precsim b^{\oplus k}$, we have $a\precsim b$.

\begin{remark}
	For a $C^*$-algebra $A$, we say $a, b \in M_\infty(A)$ are \emph{Cuntz equivalent} and write $a \sim b$ if $a \precsim b$ and $b \precsim a$.  The \emph{uncompleted Cuntz semigroup} of $A$ is the set $W(A) = M_\infty(A) / \sim$, viewed as an ordered abelian semigroup with the addition and order induced by direct sum and Cuntz subequivalence, respectively.  Then $W(A)$ is \emph{almost unperforated} if for all integers $k \geq 1$ and $x, y \in W(A)$, we have $(k+1) x \leq ky$ implies $x \leq y$.
	
	The \emph{completed Cuntz semigroup} $\mathrm{Cu}(A)$ is defined analogously using $A \otimes \mathcal K$ in place of $M_\infty(A)$, where $\mathcal K$ denotes the $C^*$-algebra of compact operators on a separable, infinite-dimensional Hilbert space and the direct sum operation is defined using a fixed identification $\mathcal K \cong M_2(\mathcal K)$.  It is a standard fact that $W(A)$ is almost unperforated if and only if $\mathrm{Cu}(A)$ is almost unperforated, so there is no ambiguity when we say that a $C^*$-algebra has almost unperforated Cuntz semigroup.  Indeed, the forward direction is \cite[Lemma~2.10]{OrtegaPereraRordam:IMRN}, applied with $n = 0$, and the converse is immediate after identifying $M_\infty(A)$ with a dense subalgebra of $A \otimes \mathcal K$.
\end{remark}

For a $C^*$-algebra $A$, a trace $\tau$ on $A$, and positive $a \in M_\infty(A)$, define
\begin{equation}
	d_\tau(a) = \lim_{n \rightarrow \infty} \tau(a^{1/n}).
\end{equation}
The function $d_\tau$ is called the \emph{dimension function} associated to $\tau$.  Note that if $a$ and $b$ are positive matrices over $A$, then $a \precsim b$ implies $d_\tau(a) \leq d_\tau(b)$ for all traces $\tau$ on $A$.  A unital $C^*$-algebra $A$ has \emph{strict comparison (with respect to traces)} if for all positive matrices $a, b \in M_\infty(A)$, $d_\tau(a) < d_\tau(b)$ for all traces $\tau$ on $A$ implies $a \precsim b$.

A connection between strict comparison and almost unperforation of the Cuntz semigroup is given by the following well-known (and easy) fact.  More generally, almost unperforation of the Cuntz semigroup is equivalent to a variation of strict comparison which accounts for all lower-semicontinuous quasitracial weights on the $C^*$-algebra; see \cite[Corollary~4.7]{Rordam:IntJMath}.  Results of this form are tricky to track down in the literature and often have superfluous hypotheses, largely due to the many competing definition of strict comparison, so we include a short proof for completeness.

\begin{proposition}\label{prop:comparison}
	If $A$ is a simple $C^*$-algebra with strict comparison, then the Cuntz semigroup of $A$ is almost unperforated.
\end{proposition}

\begin{proof}
	Suppose $a, b \in M_\infty(A)$ are positive, $k \geq 1$ is an integer, and $a^{\oplus k+1} \precsim b^{\oplus k}$.  If $b = 0$, we have $a = 0$, so $a \precsim b$.  If $b$ is non-zero, then since traces on simple $C^*$-algebras are faithful, we have $d_\tau(b) > 0$.  Further, $(k+1) d_\tau(a) \leq k d_\tau(b)$ and hence $d_\tau(a) < d_\tau(b)$ for all traces $\tau$ on $A$.  So strict comparison implies $a \precsim b$.
\end{proof}

Recall that a unital $C^*$-algebra $A$ has \emph{stable rank one} if a dense set of elements of $A$ are invertible (see \cite{Reiffel:PLMS}).  We recall here a connection between strict comparison and stable rank recently obtained by Lin in \cite{Lin:sr1} in the special case of $C^*$-algebras with unique trace.  Since this corollary is not explicilty stated in \cite{Lin:sr1}, we include a proof showing how to extract the result from \cite{Lin:sr1}.

\begin{theorem}[{\cite[Theorem~1.1]{Lin:sr1}}]\label{thm:sr1}
    If $A$ is a unital, simple $C^*$-algebra with unique trace and strict comparision, then $A$ has stable rank one.
\end{theorem}

\begin{proof}
    We may assume $A$ is infinite-dimensional since it's easy to see that $M_n(\mathbb C)$ has stable rank one for all $n \geq 1$.  By \cite[Theorem~3.6(1)]{NgRobert:MJM}, every quasitrace on $A$ is a trace, and then since $A$ has unique quasitrace $\tau$ it suffices to show $d_\tau \colon \mathrm{Cu}(A) \rightarrow [0, \infty]$ is surjective by \cite[Theorem~1.1]{Lin:sr1}.    Using \cite[Lemma~4.1]{DadarlatToms:JFA}, for example, $d_\tau$ has dense range.  It's clear that $0$ is in the range of $d_\tau$ since $0 = d_\tau(0)$.  For $t \in (0, \infty]$ fix a sequence $(x_n)_{n=1}^\infty \subseteq \mathrm{Cu}(A)$ such that $(d_\tau(x_n))_{n=1}^\infty$ is a strictly increasing sequence converging to $t$.  Since $A$ has strict comparison and $d_\tau(x_n) < d_\tau(x_{n+1})$, we have $x_n \leq x_{n+1}$.  Hence we may define $x = \sup_{n \geq 1} x_n \in \mathrm{Cu}(A)$.  Then $d_\tau(x) = t$.
\end{proof}

The following standard separabilization argument will allow us to replace the highly non-separable trace-kernel extension recalled in Section~\ref{sec:trace-kernel} with a separable subextension while maintaining the relevant properties.  This is necessary for the extension theoretic arguments in the proof of Lemma~\ref{lem:ext}, which require separability.

\begin{proposition}\label{prop:sep-au}
	Almost unperforation of the Cuntz semigroup is a separably inheritable property of $C^*$-algebras in the sense of \cite[Section~II.8.5]{Blackadar:Book}.  That is,
	\begin{enumerate}
		\item if $(A_n)_{n=1}^\infty$ is an increasing sequence of separable $C^*$-algebras with almost unperforated Cuntz semigroups, then the inductive limit has almost unperforated Cuntz semigroup, and
		\item if $A$ is a $C^*$-algebra with almost unperforated Cuntz semigroup and $S \subseteq A$ is a separable set, then there is a separable $C^*$-algebra $A_0 \subseteq A$ containing $S$ such that $A_0$ has almost unperforated Cuntz semigroup.
	\end{enumerate}
\end{proposition}

\begin{proof}
	The second part follows from the fact that almost unperforation of the Cuntz semigroup is an elementary property of $C^*$-algebras; see \cite[Proposition~8.1.6]{FHLRTVW:MAMS}.  The first part is certainly known, but we have been unable to find a suitable reference, so we give a proof here.
	
	Let $A$ denote the inductive limit of the $A_n$.  Fix an integer $k \geq 1$ and positive $a, b \in M_\infty(A)$ such that $a^{\oplus k+1} \precsim b^{\oplus k}$ in $M_\infty(A)$.  It suffices to show that $(a - 4\epsilon)_+ \precsim b$ in $M_\infty(A)$ for all $\epsilon > 0$ since this readily implies $a \precsim b$; indeed, if $v \in M_\infty(A)$ with $\|v^*b v - (a - 4\epsilon)_+ \| < \epsilon$, then $\|v^*b v - a \| < 5 \epsilon$.
	
	We will quote the functional calculus result of \cite[Section~2]{Rordam:JFA}, which is stated in terms of the continuous function $f_\epsilon \colon [0, \infty) \rightarrow \mathbb R$ given by
	\begin{equation}\label{eq:cutdown}
		f_\epsilon(t) = 
		\begin{cases} 
			0, & 0 \leq t \leq \epsilon ;\\ 
			\epsilon^{-1} (t - \epsilon), & \epsilon < t < 2\epsilon; \\
			1, & t \geq 2 \epsilon.
		\end{cases}
	\end{equation}
	Note that for any $C^*$-algebra $D$, positive $a \in D$, and $\epsilon > 0$, $f_\epsilon(a)$ is Cuntz equivalent to $(a - \epsilon)_+$ in $D$.  So when we apply the Cuntz subequivalence results from \cite{Rordam:JFA}, we will freely use $(a - \epsilon)_+$ in place of $f_\epsilon(a)$.
	
	Fix $\epsilon > 0$ and use \cite[Proposition~2.4]{Rordam:JFA}  to find $\delta > 0$ such that 
	\begin{equation}
		(a - \epsilon)_+^{\oplus k+1} \precsim (b - 2\delta)_+^{\oplus k}.
	\end{equation}
	Then choose an integer $n \geq 1$ and $a_0, b_0 \in M_\infty(A)$ such that $\|a - a_0\| < \epsilon$ and $\|b - b_0 \|< \delta$.  Then in $M_\infty(A)$, we have
	\begin{equation}
		(a_0 - 2\epsilon)_+^{\oplus k+1} \precsim (a - \epsilon)_+^{\oplus k+1} \precsim (b - 2\delta)_+^{\oplus k} \precsim (b_0 - \delta)_+^{\oplus k},
	\end{equation}
	using \cite[Proposition~2.2]{Rordam:JFA} for the first and third subequivalences.  Fix $v \in M_\infty(B)$ such that
	\begin{equation}
		\|v^*(b_0 - \delta)_+^{\oplus k} v - (a_0 - 2\epsilon)_+^{\oplus k+1} \| < \epsilon.
	\end{equation}
	Enlarging $n$ if necessary and approximating $v$ by an element of $M_\infty(A_n)$, we may assume $v \in M_\infty(A_n)$.  Then another application of \cite[Proposition~2.2]{Rordam:JFA} yields
	\begin{equation}
		(a_0 - 3\epsilon)_+^{\oplus k+1} \precsim v^*(b_0 - \delta)_+^{\oplus k}v \precsim (b_0 - \delta)_+^{\oplus k}
	\end{equation}
	in $M_\infty(A_n)$.  Since $A_n$ has almost unperforated Cuntz semigroup, it follows that $(a_0 - 3\epsilon)_+ \precsim (b_0 - \delta)_+$ in $M_\infty(A_n)$ and hence also in $M_\infty(A)$.  A final two application of \cite[Proposition~2.2]{Rordam:JFA} then yields
	\begin{equation}
		(a - 4\epsilon)_+ \precsim (a_0 - 3\epsilon)_+ \precsim (b_0 - \delta)_+ \precsim b
	\end{equation}
	in $M_\infty(A)$, completing the proof.
\end{proof}

\subsection{Ultraproducts and the $K_1$-group}\label{sec:ultraproducts}

For a sequence of $C^*$-algebras $(A_n)_{n=1}^\infty$, let $\prod A_n$ denote the $\ell^\infty$-product, and define
\begin{equation} 
	\sum_\omega A_n = \Big\{(a_n)_{n=1}^\infty \in \prod A_n : \lim_{n \rightarrow \omega} \|a_n\| = 0 \Big\},
\end{equation}
which is an ideal in $\prod A_n$.  The quotient $\prod_\omega A_n = \prod A_n / \sum_\omega A_n$ is the $C^*$-ultraproduct of the $A_n$.  When $(A_n)_{n=1}^\infty$ is the constant sequence $A$, we will instead use the notation $A_\omega$ and refer to this $C^*$-algebra as the $C^*$-ultrapower of $A$, and we identify $A$ with the subalgebra of $A_\omega$ consisting of the equivalence classes of constant sequences in $A$.  For an integer $d \geq 1$, we will freely identify $M_d(\prod A_n)$ with $\prod M_d(A_n)$, and similarly for $\sum_\omega A_n$ and $\prod_\omega A_n$.

In a similar fashion, for a sequence of abelian groups $(G_n)_{n=1}^\infty$, let $\prod G_n$ denote the product group and consider the subgroup
\begin{equation}
	\sum_\omega G_n = \Big\{(g_n)_{n=1}^\infty \in \prod G_n : g_n = 0 \text{ for $\omega$-many } n \}.
\end{equation}
The quotient $\prod_\omega G_n = \prod G_n / \sum_\omega G_n$ is the (abelian group) ultraproduct of the $G_n$.  Again, the ultrapower $G_\omega$ of an abelian group $G$ is defined as the ultraproduct of the constant sequence $G$.

For a sequence of $C^*$-algebras $(A_n)_{n=1}^\infty$, there is a group homomorphism 
\begin{equation}\label{eq:eta}
	\eta \colon K_1\big(\prod A_n\big) \rightarrow \prod K_1(A_n) \colon \big[(u_n)_{n=1}^\infty\big]_1 \mapsto \big([u_n]_1\big)_{n=1}^\infty
\end{equation}
induced by the projection maps $\prod A_n \twoheadrightarrow A_m$ for $m \geq 1$.  The map $\eta$ canonically induces a map 
\begin{equation}
	\bar\eta \colon K_1\big(\prod_\omega A_n\big) \rightarrow \prod_\omega K_1(A_n)
\end{equation}
as we now describe.  

Consider the exact sequence 
\begin{equation}
	\begin{tikzcd}
		0 \arrow{r} & \sum_\omega A_n \arrow{r}{\iota} & \prod A_n \arrow{r}{\rho} & \prod_\omega A_n \arrow{r} & 0
	\end{tikzcd}
\end{equation}
of $C^*$-algebras.  Since projections (resp. unitaries) in $M_d(\prod_\omega A_n)$ lift to projections (resp.\ unitaries) in $M_d(\prod A_n)$ for all $d \geq 1$, the map $K_0(\rho)$ (resp.\ $K_1(\rho))$ is surjective.  So the six-term exact sequence in $K$-theory induces the short exact sequence in the top row of the diagram
\begin{equation}\label{eq:k1-ultraproduct}
	\begin{tikzcd}
		0 \arrow{r} & K_1\big(\sum_\omega A_n\big) \arrow{r}{K_1(\iota)} \arrow[dashed]{d}{\eta_0} & K_1\big(\prod A_n \big) \arrow{r}{K_1(\rho)} \arrow{d}{\eta} & K_1\big(\prod_\omega A_n\big) \arrow{r} \arrow[dashed]{d}{\bar\eta} & 0\phantom{.} \\
		0 \arrow{r} & \sum_{\omega} K_1(A_n) \arrow{r} & \prod K_1(A_n) \arrow{r} & \prod_\omega K_1(A_n) \arrow{r} & 0.
	\end{tikzcd}
\end{equation}
Fix an integer $d \geq 1$ and a unitary $u = (u_n)_{n=1}^\infty \in M_d(\sum_\omega A_n + \mathbb C1_{\prod A_n})$.  There is a sequence of unitaries $(v_n)_{n=1}^\infty \in M_d(\mathbb C1_{\prod A_n})$ such that $\|u_n - v_n\| \rightarrow 0$ along $\omega$.  Then for $\omega$-many $n$, we have $\|u_n - v_n\| < 2$, and for all such $n$, $[u_n]_1 = [v_n]_1 = 0 \in K_1(A_n)$.  So $\eta([u]_1) \in \sum_\omega K_1(A_n)$, giving the existence of the map $\eta_0$ and hence also $\bar\eta$.

\subsection{The trace-kernel extension}\label{sec:trace-kernel}

We will also need to consider the tracial ultrapower of a tracial $C^*$-algebra.  More precisely, let $A$ be a $C^*$-algebra with a fixed trace $\tau$ and equip $A$ with the seminorm $\|a\|_2 = \tau(a^*a)^{1/2}$ for $a \in A$.  (When this construction is applied, $A$ will have unique trace, so we do not include $\tau$ in the notation.)  The \emph{tracial ultrapower} of $A$ is the $C^*$-algebra
\begin{equation}
	A^\omega = \ell^\infty(A) /\{ (a_n)_{n=1}^\infty \in \ell^\infty(A) : \lim_{n \rightarrow \omega} \|a_n\|_2 = 0 \},
\end{equation}
where $\ell^\infty(A)$ denotes the $C^*$-algebra of bounded sequences in $A$.  Since $\|a\|_2 \leq \|a\|$ for all $a \in A$, there is a canonical quotient map $q_A \colon A_\omega \rightarrow A^\omega$ induced by the identify map on representing sequences.  The \emph{trace-kernel ideal} of $A$ is the ideal $J_A = \ker(q_A) \subseteq A_\omega$.

It is well-known that $A^\omega$ is always a tracial von Neumann algebra with a faithful normal trace defined on representing sequences by
\begin{equation}\label{eq:limit-trace}
	(a_n)_{n=1}^\infty \mapsto \lim_{n \rightarrow \omega} \tau(a_n).
\end{equation}
In fact, if $\pi_\tau$ denotes the GNS representation of $A$, then $A^\omega$ is precisely the tracial ultrapower of the tracial von Neumann algebra $\pi_\tau(A)''$; indeed, Kaplansky's density theorem implies the map $A^\omega \rightarrow (\pi_\tau(A)'')^\omega$ obtained by applying $\pi_\tau$ entrywise is an isomorphism.  If $\tau$ is the unique trace on $A$, then $\pi_\tau(A)''$ is a factor (see \cite[Theorem~6.7.4 and Corollary~6.8.5]{Dixmier:Book}) and hence so is $A^\omega$.  When we apply this construction in the proof of Theorem~\ref{thm:k1-inj}, $A$ will be a simple, infinite-dimensional $C^*$-algebra with unique trace, and then $A^\omega$ will be a II$_1$ factor; indeed, since $A$ is simple, the canonical map $A \rightarrow A^\omega$ is injective, and since $A$ is infinite-dimensional, so is $A^\omega$.  We use the same symbol $\tau$ for the traces on $A_\omega$ and $A^\omega$ defined on representing sequences by \eqref{eq:limit-trace}.

The following result collects the properties of $J_A$ which will be needed in the next section.  
Recall from \cite[Definition~1.4]{Schafhauser:AnnMath} that a $C^*$-algebra $I$ is \emph{separably stable} if every separable $C^*$-subalgebra of $I$ is contained in a separable, stable $C^*$-subalgebra of $I$.  The proof of separable stability of $J_A$ is essentially the same as that given in \cite[Lemma~6.11]{CGSTW1}, specialized to the case $A$ has unique trace.  The result in \cite[Lemma~6.11]{CGSTW1} also assumes $A$ is exact and $\mathcal Z$-stable, but this is only needed to access strict comparison.  A more general result of this form will appear in \cite{CGSTW2}.
	
\begin{proposition}[{cf.\ \cite[Lemma~6.11]{CGSTW1}}]\label{prop:trace-kernel-ideal}
	If $A$ is a unital, simple $C^*$-algebra with unique trace and strict comparison, then $J_A$ is separably stable and has almost unperforated Cuntz semigroup.
\end{proposition}

\begin{proof}
	First note that $A$ has almost unperforated Cuntz semigroup by Proposition~\ref{prop:comparison}.  Further, the condition is elementary, and in particular passes to ultrapowers, by \cite[Proposition~8.1.6]{FHLRTVW:MAMS}, so $A_\omega$ has almost unperforated Cuntz semigroup.  It follows that $J_A$ has almost unperforated Cuntz semigroup as for any positive $a, b \in M_\infty(J_A)$, $a \precsim b$ in $M_\infty(A_\omega)$ implies $a \precsim b$ in $M_\infty(J_A)$.  Indeed, if $a \precsim b$ in $M_\infty(A_\omega)$ and $\epsilon > 0$, then there exists $v \in M_\infty(A_\omega)$ with $\|v^*b v - a\| < \epsilon$.  Choose a positive contraction $e \in M_\infty(J_A)$ with $\|e b - b\| < \epsilon / (\|v\|^2 + 1)$, and note that $ev \in M_\infty(J_B)$ satisfies 		$\|(ev)^*b(ev) - a\| < 3 \epsilon$.	
		
	To show $J_A$ is separably stable, it suffices to show that for all $\epsilon > 0$ and positive $a \in J_A$, there is $v \in J_A$ such that $v^*v = (a - 2\epsilon)_+$ and $v^* (a - 2\epsilon)_+ = 0$ by \cite[Proposition~6.10]{CGSTW1}, which is essentially a restatement of the Hjelmborg--R{\o}rdam characterization of stability for $\sigma$-unital $C^*$-algebras in \cite{HjelmborgRordam:JFA}.  Fix such an $\epsilon$ and $a$, and let $f_\epsilon$ be the function defined in \eqref{eq:cutdown}.  Since $a, f_\epsilon(a) \in J_A$, we have
	\begin{equation}
		d_\tau(a) = 0 < 1 = d_\tau(1_{A_\omega} - f_\epsilon(a)).
	\end{equation}
	Also, $A_\omega$ has strict comparison with respect to the trace $\tau$ by \cite[Lemma~1.23]{BBSTWW:MAMS}, so $a \precsim 1_{A_\omega} - f_\epsilon(a)$ in $A_\omega$.  By \cite[Proposition~2.7(iii)]{KirchbergRordam:AmerJMath}, there exists $v \in A_\omega$ such that $v^*v = (a - 2\epsilon)_+$ and $vv^*$ is in the hereditary subalgebra of $A_\omega$ generated by $1_{A_\omega} - f_\epsilon(a)$.  Note that $v \in J_A$ since $v^*v \in J_A$, and $v^*(a - 2\epsilon)_+ = 0$  since $(1_{A_\omega} - f_\epsilon(a)) (a - 2\epsilon)_+ = 0$.
\end{proof}
	
\subsection{Absorption}\label{sec:absorption}

The proof of Theorem~\ref{thm:k1-inj} will combine the von Neumann algebraic structure of $A^\omega$ with certain rigidity results about (separable subextensions of) the trace-kernel extension coming from the Elliott--Kucerovsky characterization of absorbing extensions in \cite{ElliottKucerovsky:PJM} and the regularity results of Kucerovsky and Ng in \cite{KucerovskyNg:HJM} and Ortega, Perera, and R{\o}rdam in \cite{OrtegaPereraRordam:IMRN}.  We briefly recall the theorem in the form that will be needed here.

Recall that an element $a$ is a $C^*$-algebra $A$ is \emph{full} if $a$ generates $A$ as a (closed, two-sided) ideal, and a $^*$-homomorphism $\phi \colon A \rightarrow B$ between $C^*$-algebras $A$ and $B$ is \emph{full} if $\phi(a)$ is full in $B$ for every non-zero $a \in A$.  For a $C^*$-algebra $I$, let $\mathcal M(I)$ denote the multiplier algebra of $I$.  A $\sigma$-unital, stable $C^*$-algebra $I$ has the \emph{corona factorization property}, introduced in \cite{KucerovskyNg:HJM}, if every full projection in $\mathcal M(I)$ is properly infinite.  The corona factorization property was shown to hold very generally by Ortega, Perera, and R{\o}rdam in \cite{OrtegaPereraRordam:IMRN}.  For us, the main examples will be separable, stable $C^*$-algebras with almost unperforated Cuntz semigroup (see \cite[Propostion~2.17 and Thoerem~5.11]{OrtegaPereraRordam:IMRN}).  In particular, Propositions~\ref{prop:sep-au} and~\ref{prop:trace-kernel-ideal} imply that when $A$ is a unital, simple $C^*$-algebra with unique trace and strict comparison, the $C^*$-algebra $J_A$ has arbitrarily large separable subalgebras with the corona factorization property.

The utility of the corona factorization property is its relation to absorption. Let $A$ and $I$ be $C^*$-algebras with $A$ separable and unital and $I$ $\sigma$-unital and stable.  A unital $^*$-homomorphism $\psi \colon A \rightarrow \mathcal M(I)$ is called \emph{unitally absorbing} if for every unital $^*$-homomorphism $\phi \colon A \rightarrow \mathcal M(I)$, there is a sequence of $2 \times 1$ matrices $(v_n)_{n=1}^\infty$ over $\mathcal M(I)$ such that
\begin{align}
	\| &v_n^* (\phi(a) \oplus \psi(a)) v_n - \psi(a) \| \rightarrow 0, &&a \in A,
\shortintertext{and}
	&v_n^* (\phi(a) \oplus \psi(a)) v_n - \psi(a) \in I, &&a \in A,\ n \geq 1.
\end{align}
The following result, reminiscent of Voiculescu's theorem, provides a powerful method for verifying absorption.  The result does not seem to appear in the literature in precisely this form, but it is a simple repackaging of known results; variations of this have also been recorded as \cite[Theorem~2.3]{Schafhauser:AnnMath} and \cite[Theorem~5.13]{CGSTW1}, for example.
	
\begin{theorem}[see \cite{ElliottKucerovsky:PJM} and \cite{KucerovskyNg:HJM}]\label{thm:absorption}
	If $A$ is a separable, unital, nuclear $C^*$-algebra, $I$ is a $\sigma$-unital, stable $C^*$-algebra with the corona factorization property, and $\psi \colon A \rightarrow \mathcal M(I)$ is a unital, full $^*$-homomorphism, then $\psi$ is unitally absorbing.
\end{theorem}

\begin{proof}
	Let $q \colon \mathcal M(I) \rightarrow \mathcal M(I) / I$ denote the quotient map.  Since $\psi$ is full and $q$ is unital, $q \circ \psi \colon A \rightarrow \mathcal M(I) / I$ is full.  The conditions on $A$ and $I$ imply that $q \circ \psi$ is absorbing in the sense that for all unital $^*$-homomorphisms $\phi \colon A \rightarrow \mathcal M(I)$, there is a $2 \times 1$ matrix $v$ over $\mathcal M(I) / I$ such that 
	\begin{equation} 
		v^* (q(\phi(a)) \oplus q(\psi(a))) v = \psi(a), \qquad a \in A,
	\end{equation}
	by the equivalence of (i) and (v) in \cite[Theorem~3.5]{KucerovskyNg:HJM}, which depends heavily on the characterization of (nuclear) unital absorption in \cite{ElliottKucerovsky:PJM}.  It follows that $\psi$ is unitally absorbing; see the proof of \cite[Corollary~5.10]{CGSTW1}, for example (which proves the non-unital case).
\end{proof}

Absorption will be applied through the following \emph{stable uniqueness theorem} of Dadarlat and Eilers (\cite{DadarlatEilers:KT}), which has played a fundamental role in the classification theory for simple, nuclear $C^*$-algebras.  For a $\sigma$-unital, stable $C^*$-algebra $I$, two $^*$-homomorphisms $\phi_0, \phi_1 \colon A \rightarrow \mathcal M(I)$ are \emph{properly asymptotically unitarily equivalent}, written $\phi_0 \approxeq \phi_1$, if there is a continuous path of unitaries $(w_t)_{t \geq 0} \subseteq U(I + \mathbb C 1_{\mathcal M(I)})$ such that 
\begin{equation}
	\lim_{t \rightarrow \infty} \|\phi_0(a) - w_t \phi_1(a) w_t^* \| = 0, \qquad a \in A.
\end{equation}
Note that in the conclusion of the theorem, we are implicitly identifying $M_2(\mathcal M(I))$ with $\mathcal M(M_2(I))$.  The following result is known, and while it is not explicitly stated in \cite{DadarlatEilers:KT}, it follows readily from the results therein.

\begin{theorem}[{cf.\ \cite[Theorem~3.8]{DadarlatEilers:KT}}]\label{thm:stable-uniqueness}
	Let $A$ be a separable, unital $C^*$-algebra and let $I$ be a $\sigma$-unital, stable $C^*$-algebra.  Suppose $\phi_0, \phi_1, \psi \colon A \rightarrow \mathcal M(I)$ are unital $^*$-homomorphisms such that $\operatorname{im}(\phi_0 - \phi_1) \subseteq I$ and $\psi$ is unitally absorbing.  If $[\phi_0, \phi_1] = 0 \in KK(A, I)$, then $\phi_0 \oplus \psi \approxeq \phi_1 \oplus \psi$ as maps into $M_2(\mathcal M(I))$.
\end{theorem}

\begin{proof}
	Let $\psi_\infty \colon A \rightarrow \mathcal M(I)$ denote the infinite repeat of $\psi$; this is the composition
	\begin{equation}
	\begin{tikzcd}[column sep = small]
		A \arrow{r}{\psi \otimes 1_{\mathcal M(\mathcal K)}} &[4.5ex] \mathcal M(I) \otimes \mathcal M(\mathcal K) \arrow{r} & \mathcal M(I \otimes \mathcal K) \arrow{r}{\cong} & \mathcal M(I),
	\end{tikzcd}
	\end{equation}
where the second map is the canonical inclusion and the third map is induced by an isomorphism $I \otimes \mathcal K \overset\cong\rightarrow I$, using the stability of $I$.  Then $\phi_0 \oplus \psi_\infty \approxeq \phi_1 \oplus \psi_\infty$ by \cite[Theorem~3.8]{DadarlatEilers:KT}.  As $\psi$ is unitally absorbing, so is $\psi_\infty$.  Then two applications of \cite[Lemma~2.3]{DadarlatEilers:KT} (once with $\sigma = \psi_\infty$ and $\pi = \psi$ and once with $\sigma = \psi$ and $\pi = \psi_\infty)$ imply $\psi \sim_{\rm asymp} \psi \oplus \psi_\infty \sim_{\rm asymp} \psi_\infty$ in the sense of \cite[Definition~2.1]{DadarlatEilers:KT}, and the result follows from \cite[Lemma~3.4]{DadarlatEilers:KT}.
\end{proof}

\section{Proof of Theorem~\ref{thm:k1-inj}}

Before proceeding to the proof of Theorem~\ref{thm:k1-inj}, we isolate the following lemma to separate the general extension theoretic methods from the specific structure of the trace-kernel extension.  Note that if $I$ is an ideal in a unital $C^*$-algebra $E$ and $u_0, u_1 \in E$ are unitaries with $u_0- u_1 \in I$, then $u_0u_1^*$ is a unitary in $I + \mathbb C 1_E$.  Further, if $[u_0u_1^*]_1 = 0 \in K_1(I)$, after passing to large matrices, $u_0$ and $u_1$ are homotopic via a homotopy which is constant modulo $I$.  The following result will allow us to both control the size of the matrix stabilization and replace homotopy with a more rigid equivalence relation.  We say that a unitary $v$ in a unital $C^*$-algebra $E$ is \emph{totally full} if the unital $^*$-homomorphism $C(\mathbb T) \rightarrow E \colon f \mapsto f(v)$ is full.

\begin{lemma}\label{lem:ext}
	Let $0 \rightarrow I \overset j \rightarrow E \overset q \rightarrow D \rightarrow 0$ be a unital extension of $C^*$-algebras such that $I$ is separably stable with almost unperforated Cuntz semigroup.  Further, let $u_0, u_1, v \in E$ be unitaries such that $v$ is totally full and $q(u_0) = q(u_1)$.  If $[u_0u_1^*]_1 = 0 \in K_1(I)$, then there is a continuous path of unitaries $(w_t)_{t \geq 0} \subseteq M_2(I + \mathbb C1_E) \subseteq M_2(E)$ such that
	\begin{equation}\label{eq:ext}
		\lim_{t \rightarrow \infty} \Big\| \begin{pmatrix} u_0 & 0 \\ 0 & v \end{pmatrix} - w_t \begin{pmatrix} u_1 & 0 \\ 0 & v \end{pmatrix} w_t^* \Big\| = 0.
	\end{equation}
\end{lemma}

\begin{proof}
We first reduce to the case that the extension is separable.  Applying \cite[Proposition~1.9]{Schafhauser:AnnMath} (or rather, its proof) to the unital $^*$-homomorphism $C(\mathbb T) \rightarrow E \colon f \mapsto f(v)$, there is a separable $C^*$-algebra $E_0 \subseteq E$ such that $u_0, u_1, v \in E_0$ and $v$ is totally full in $E_0$.  Further, using \cite[Proposition~2.10]{Schafhauser:AnnMath}, for example, there is a separable $C^*$-algebra $I_0 \subseteq I$ such that $u_0u_1^* \in I_0 + \mathbb C1_E$ and $[u_0u_1^*]_1 = 0 \in K_0(I_0)$.  By Proposition~\ref{prop:sep-au}, almost unperforation of the Cuntz semigroup is a separably inheritable property, and by \cite[Corollary~4.1]{HjelmborgRordam:JFA}, stability is preserved under sequential direct limits of separable $C^*$-algebras.  So by \cite[Proposition~1.6]{Schafhauser:AnnMath}, together with \cite[Proposition~1.5]{Schafhauser:AnnMath}, there is a separable subextension
\begin{equation}\label{eq:sep-subext}
\begin{tikzcd}
	0 \arrow{r} & I_1 \arrow{r}{j_1} \arrow[tail]{d} & E_1 \arrow{r}{q_1} \arrow[tail]{d} & D_1 \arrow{r} \arrow[tail]{d} & 0 \\
	0 \arrow{r} & I \arrow{r}{j} & E \arrow{r}{q} & D \arrow{r} & 0
\end{tikzcd}
\end{equation}
such that $I_0 \subseteq I_1$, $I_1$ is stable with almost unperforated Cuntz semigroup, and $E_0 \subseteq E_1$. After replacing the given extension with the top row of \eqref{eq:sep-subext}, we may assume $I$, $E$, and $D$ are separable.

Define unital $^*$-homomorphisms $\phi_0, \phi_1, \psi \colon C(\mathbb T) \rightarrow  E$ by
\begin{equation}
	\phi_0(f) = f(u_0), \quad \phi_1(f) = f(u_1), \quad \text{and} \quad \psi(f) = f(v)
\end{equation}
for $f \in C(\mathbb T)$.  Note that $q \circ \phi_0 = q \circ \phi_1$ since $q(u_0) = q(u_1)$ and $\psi$ is full since $v$ is totally full.  Let $\mathcal M(I)$ denote the multiplier algebra of $I$ and let $\lambda \colon E \rightarrow \mathcal M(I)$ denote the canonical map.  Then $\operatorname{im}(\lambda \circ \phi_0 - \lambda \circ \phi_1) \subseteq I$, and hence $\lambda \circ \phi_0$ and $\lambda \circ \phi_1$ induce an element $[\lambda \circ \phi_0, \lambda \circ \phi_1] \in KK(C(\mathbb T), I)$.

Since $[u_0u_1^*]_1 = 0 \in K_1(I)$, there is an integer $d \geq 0$ and a path of unitaries $(\tilde u_t)_{0 \leq t \leq 1}$ of unitaries in $M_{d+1}(E)$ such that $\tilde u_i = u_i \oplus 1_E^{\oplus d}$ for $i = 0, 1$ and $q(\tilde u_t) = q(u_0) \oplus 1_D^{\oplus d}$ for all $t \in [0, 1]$.  Define $\chi \colon C(\mathbb T) \rightarrow \mathcal M(I)$ by $\chi(f) = f(1)1_{\mathcal M(I)}$.  Then the path of $^*$-homomorphisms
\begin{equation}
	C(\mathbb T) \rightarrow M_{d+1}(\mathcal M(I)) \colon f \mapsto f(\lambda(\tilde u_t)), \qquad t \in [0, 1],
\end{equation}
defines a homotopy from $(\lambda \circ \phi_0) \oplus \chi^{\oplus d}$ to $(\lambda \circ \phi_1) \oplus \chi^{\oplus d}$ which is constant modulo $M_{d+1}(I)$.  So, \begin{equation}
	[\lambda \circ \phi_0, \lambda \circ \phi_1] = [(\lambda \circ \phi_0) \oplus \chi^{\oplus d}, (\lambda \circ \phi_1) \oplus \chi^{\oplus d}] = 0. 
\end{equation}

Since $I$ is stable and has almost unperforated Cuntz semigroup, \cite[Propostion~2.17 and Thoerem~5.11]{OrtegaPereraRordam:IMRN} imply $I$ has the corona factorization property, so $\psi$ is unitally absorbing by Theorem~\ref{thm:absorption}.  Then $\phi_0 \oplus \psi \approxeq \phi_1 \oplus \psi$ by Theorem~\ref{thm:stable-uniqueness}.  Fix a continuous path of unitaries $(\bar w_t)_{t \geq 0}$ in $M_2(I + \mathbb C1_{\mathcal M(I)})$ such that
\begin{equation}
	\lim_{t \rightarrow \infty} \Big\| \begin{pmatrix} \lambda(\phi_0(f)) & 0 \\ 0 & \lambda(\psi(f)) \end{pmatrix} - \bar w_t \begin{pmatrix} \lambda(\phi_1(f)) & 0 \\ 0 & \lambda(\psi(f)) \end{pmatrix}\bar w_t^*\Big\| = 0.
\end{equation}
for all $f \in C(\mathbb T)$.  Taking $f \in \mathbb C(\mathbb T)$ to be the function $f(z) = z$ for $z \in \mathbb T$, we have
\begin{equation}
	\lim_{t \rightarrow \infty} \Big\| \begin{pmatrix} \lambda(u_0) & 0 \\ 0 & \lambda(v) \end{pmatrix} - \bar w_t \begin{pmatrix} \lambda(u_1) & 0 \\ 0 & \lambda(v) \end{pmatrix}\bar w_t^* \Big\| = 0.
\end{equation}
Let $w_t \in M_2(I + \mathbb C1_E)$ be the unitary with $\lambda(w_t) = \bar w_t$ for $t \geq 0$.  Since $\lambda$ is isometric on $M_2(I)$, \eqref{eq:ext} follows.
\end{proof}

We now have the machinery needed to prove Theorem~\ref{thm:k1-inj} from the introduction.  Proposition~\ref{prop:trace-kernel-ideal} shows that the trace-kernel extension satisfies the conditions of Lemma~\ref{lem:ext}.  The basic strategy is to start with a unitary $u \in A$ with $[u]_1 = 0 \in K_1(A)$ and use the von Neumann algebraic structure of $A^\omega$ to find two self-adjoints $a, b \in A_\omega$ such that $u_0 = u$ and $u_1 = e^{ia} e^{ib}$ satisfy $q_A(u_0) = q_A(u_1)$ and $[u_0u_1^*] = 0 \in K_1(J_A)$.  If the direct summand $v$ in Lemma~\ref{lem:ext} could be removed, it would follow that $u$ is a product of two exponentials in $A_\omega$.  While we do not know how to remove this summand, taking $v = e^{ic}$ for a suitable self-adjoint $c$ will show $u \oplus e^{ic}$ is unitarily equivalent to $e^{ia} e^{ib} \oplus e^{ic}$ in $M_2(A_\omega)$, and rearranging gives that $u \oplus 1_{A_\omega}$ is a product of three exponentials.  Running this argument with a bit more care will further allow us to drop the number of exponentials to two and control the norm of the exponents.

\begin{proof}[Proof of Theorem~\ref{thm:k1-inj}]
	If $\sigma(u) \neq \mathbb T$, there is a scalar $t \in [-\pi, \pi]$ such that $-1 \notin \sigma(e^{-it} u)$.  Then by continuous functional calculus, there is a self-adjoint $k_0 \in A$ of norm at most $\pi$ such that $e^{-it}u = e^{i k_0}$.  The result follows by taking $h = t 1_A \oplus 1_A$ and $k = k_0 \oplus 1_A$.  So we may assume $\sigma(u) = \mathbb T$.  In particular, this implies $A$ is infinite-dimensional, and hence $A^\omega$ is a II$_1$ factor.
	
	Using Borel functional calculus, find a positive element $\bar a \in A^\omega$ of norm at most $2\pi$ such that $q_A(u) = e^{-i \bar a}$ and let $a \in A_\omega$ be a positive element of norm at most $2\pi$ with $q_A(a) = \bar a$.  Then $q_A(ue^{ia}) = 1_{A^\omega}$, so $ue^{ia}$ defines an element $[ue^{ia}]_1 \in K_1(J_A)$.  Also, since $ue^{ia}$ is homotopic to $u$ in the unitary group of $A_\omega$, $K_1(j_B)([ue^{ia}]_1) = [u]_1 = 0$ in $K_1(A_\omega)$.  Therefore, $[ue^{ia}]_1$ is in the image of the exponential map
	\begin{equation}
		\mathrm{exp} \colon K_0(A^\omega) \rightarrow K_1(J_A).
	\end{equation}
	Let $x_0 \in K_0(A^\omega)$ be such that $\mathrm{exp}(x_0) = [ue^{ia}]_1$.  Since $A^\omega$ is a II$_1$ factor, the trace $\tau$ on $A^\omega$ induces an isomorphism $\hat\tau \colon K_0(A^\omega) \rightarrow \mathbb R$.  Let $n \in \mathbb Z$ be such that $\hat\tau(x_0) + n \in [0, 1)$ and set $x = x_0 + n[1_{A^\omega}]_0 \in K_0(A^\omega)$.  Note that $q_A$ is unital and hence $[1_{A^\omega}]_0 \in \operatorname{im}( K_0(q_A)) = \ker(\mathrm{exp})$, so $\mathrm{exp}(x) = \mathrm{exp}(x_0) = [ue^{ia}]_1$.  Let $p \in A^\omega$ be a projection with $\tau(p) = \hat\tau(x)$, so $[p]_0 = x$.  Fix a positive element $b \in A_\omega$ of norm at most $2 \pi$ such that $q_A(b) = 2 \pi p$.  Then $e^{ib}$ is a unitary in the unitization of $J_A$ and $[ue^{ia}]_1 = [e^{ib}]_1 \in K_1(J_A)$.
	
	Let $c \in A$ be a positive contraction with spectrum $[0, 2\pi]$; for example, since we assumed $\sigma(u) = \mathbb T$, we may take $c = \pi(u + u^* + 1_A)$.  Then for all non-zero $f \in C(\mathbb T)$, $f(e^{ic})$ is non-zero and hence full as $A$ is simple.  So $e^{ic}$ is a totally full unitary in $A_\omega$.  Note that the trace-kernel extension
	\begin{equation}
	\begin{tikzcd}
		0 \arrow{r} & J_A \arrow{r}{j_A} & A_\omega \arrow{r}{q_A} & A^\omega \arrow{r} & 0
	\end{tikzcd}
	\end{equation}
	satisfies the conditions of Lemma~\ref{lem:ext} (see Proposition~\ref{prop:trace-kernel-ideal}).  Applying the lemma with $u_0 = ue^{ia}$, $u_1 = e^{ib}$ and $v = e^{ic}$, a standard reindexing argument in $A_\omega$ produces a unitary $w \in M_2(A_\omega)$ such that 
	\begin{equation}
		\begin{pmatrix} u e^{ia} & 0 \\ 0 & e^{ic} \end{pmatrix} = w \begin{pmatrix} e^{ib} & 0 \\ 0 & e^{ic} \end{pmatrix} w^*
	\end{equation}
	Then the elements
	\begin{equation}
		h = w \begin{pmatrix} b & 0 \\ 0 & c \end{pmatrix} w^* - \pi 1_{M_2(A_\omega)} \quad \text{and} \quad k = \pi 1_{M_2(A_\omega)} - \begin{pmatrix} a & 0 \\ 0 &  c \end{pmatrix} 
	\end{equation}
	are self-adjoint elements of $M_2(A_\omega)$ with norm at most $\pi$ such that $u \oplus 1_A = e^{ih} e^{ik}$.  The result then follows from lifting $h$ and $k$ to sequences of self-adjoint elements of norm at most $\pi$ in $M_2(A)$.
\end{proof}

\section{Proof of Theorem~\ref{thm:tarski}}

With Theorem~\ref{thm:k1-inj} in hand, we now have the machinery to prove that the $K_1$-group commutes with products and ultraproducts of unital, simple $C^*$-algebras with unique trace and strict comparsion.

\begin{theorem}\label{thm:k1-product}
	If $(A_n)_{n=1}^\infty$ is a sequence of unital, simple $C^*$-algebras with unique trace and strict comparison, then the map $\eta$ in \eqref{eq:eta} is an isomorphism.
\end{theorem}

\begin{proof}
	Each $A_n$ has stable rank one by Theorem~\ref{thm:sr1}.  It follows from \cite[Theorem~2.10]{Rieffel:JOT} that every $x \in K_1(A_n)$ is induced by a unitary $u \in A_n$.  Hence given $x = (x_n)_{n=1}^\infty \in \prod K_1(A_n)$, there is a unitary $u = (u_n)_{n=1}^\infty \in \prod A_n$ such that $[u_n]_1 = x_n \in K_1(A_n)$ for all $n \geq 1$.  Then, by construction, $\eta([u]_1) = x$, so $\eta$ is surjective.
	
	To see injectivity, suppose $d \geq 1$ and $u = (u_n)_{n=1}^\infty \in M_d\big(\prod A_n \big)$ is a unitary with $[u_n]_1 = 0$ for all $n \geq 1$.  Note that $M_d(A_n)$ is a unital, simple $C^*$-algebra with unique trace and strict comparison.  Hence by Theorem~\ref{thm:k1-inj}, for each $n \geq 1$, there are self-adjoints $h_n, k_n \in U(M_{2d}(A_n))$ with norm at most $\pi$ such that
	\begin{equation}
		\| u_n\oplus 1_{M_d(A_n)} - e^{ih_n} e^{ik_n} \| < 1.
	\end{equation}
	Set $h = (h_n)_{n=1}^\infty$ and $k = (k_n)_{n=1}^\infty \in M_{2d}\big(\prod A_n\big)$.  Then
	\begin{equation}
		\| u\oplus 1_{M_d(\prod A_n)} - e^{ih} e^{ik} \| \leq 1 < 2.
	\end{equation}  
	It follows that $u$ is homotopic to $e^{ih} e^{ik}$ and hence also to $1_{M_{2d}(\prod A_n)}$ in the unitary group of $M_{2d}\big(\prod A_n\big)$, so $[u]_1 = 0$.
\end{proof}

\begin{corollary}\label{cor:k1-ultraproduct}
	If $(A_n)_{n=1}^\infty$ is a sequence of unital, simple $C^*$-algebras with unique trace and strict comparison, then the map $\eta$ in \eqref{eq:eta} induces an isomorphism
	\begin{equation}
		\bar\eta \colon K_1\big( \prod_\omega A_n\big) \overset\cong\longrightarrow \prod_\omega K_1(A_n).
	\end{equation}
\end{corollary}

\begin{proof}
In the notation of \eqref{eq:k1-ultraproduct}, it suffices to show $\eta_0$ is an isomorphism.  The injectivity of $K_1(\iota)$ and $\eta$ implies that of $\eta_0$.  To see $\eta_0$ is surjective, fix $x = (x_n)_{n=1}^\infty \in \sum_\omega K_1(A_n)$, and using that $\eta$ is surjective, find an integer $d \geq 1$ and a unitary $u = (u_n)_{n=1}^\infty \in M_d(\prod A_n)$ such that $[u_n]_1 = x_n$ for all $n$.  Define $v = (v_n)_{n=1}^\infty \in M_d(\prod A_n)$ by $v_n = 1_{M_d(A_n)}$ when $x_n = 0$ and $v_n = u_n$ otherwise.  As $x_n = 0$ for $\omega$-many $n$, it follows that $v \in M_d(\sum_\omega A_n + \mathbb C1_{\prod A_n})$.  By construction, $\eta_0([v]_1) = x$, so $\eta_0$ is surjective.  Therefore, $\eta_0$ is an isomorphism.
\end{proof}

Corollary~\ref{cor:k1-ultraproduct} now implies the main result following the strategy outlined in the introduction.

\begin{proof}[Proof of Theorem~\ref{thm:tarski}]
We may assume $m, n \geq 2$ since it's clear that $C^*_r(F_0) \cong \mathbb C$ and $C^*_r(F_1) \cong C(\mathbb T)$ are neither elementarily equivalent to each other nor to $C^*_r(F_n)$ for $n \geq 2$.  By \cite[Corollary~3.2]{PimsnerVoiculescu:JOT}, $K_1(C^*_r(F_n)) = \mathbb Z^n$. 
Note also that $C^*_r(F_n)$ is simple and has unique trace by \cite{PaschkeSalinas:PJM} (see also \cite{Powers} for the case $n = 2$) and has strict comparision by \cite[Thoerem~A]{AGKEP} (see also \cite[Proposition~6.3.2]{Robert:AdvMath} for the case $n = \infty$).  Then Corollary~\ref{cor:k1-ultraproduct} implies
$K_1(C^*_r(F_n)_\omega) \cong (\mathbb Z^n)_\omega$.  

It suffices to show $\mathbb Z^m$ and $\mathbb Z^n$ are not elementarily equivalent when $m \neq n$.  The following sentence distinguishes $\mathbb Z^m$ from $\mathbb Z^n$ for $m > n$: 
\begin{equation}
    \exists x_1, \ldots, x_n\,  \forall y\, \exists z,\ \bigvee_{I\subseteq \{1,\ldots, n\}} \Big(\sum_{i\in I} x_i +y = 2z\Big).
\end{equation}
Indeed, this holds in $\mathbb{Z}^n$ with $x_i$ being the standard generators for $\mathbb Z^n$ by checking the parity of the coordinates of $y$.  Further, this fails in $\mathbb{Z}^m$ since  if $x_1, \ldots, x_n$ in $\mathbb{Z}^m$ satisfy the condition, then in the quotient by $(2\mathbb{Z})^m$,  $\bar x_1, \ldots, \bar x_n$ span $(\mathbb{Z}/2\mathbb{Z})^m$, which cannot happen since $(\mathbb{Z}/2\mathbb{Z})^m$ is an $m$-dimensional vector space.
\end{proof}

\bibliographystyle{amsplain}
\bibliography{tarski.bib}

\end{document}